\documentclass[aop,MSNbibl,seceqn,dvips]{arximspdf}

\doi{10.1214/13-AOP905} 
\volume{42}
\issue{4}
\pubyear{2014}
\firstpage{1699}
\lastpage{1723}

\makeatletter
\newcommand{\eqref}[1]{(\ref{#1})}
\newcommand{\eps}{\varepsilon}
\newcommand{\R}{\mathbb{R}}
\newcommand{\E}{\mathbb{E}}
\newcommand{\pr}{\mathbb{P}}
\newcommand{\1}{\mathbf{1}}
\newcommand{\del}{\partial}

\def\levy{L\'{e}vy}

\newtheorem{theorem}{Theorem}[section]

\newtheorem{lemmas}{Lemma}[section]
\newtheorem{proposition}[theorem]{Proposition}

\newproclaim{defi}[theorem]{Definition}
\newproclaim{example}[theorem]{Example}
\newproclaim{remark}[theorem]{Remark}
\makeatother

\begin{document}
\begin{frontmatter}

\title{Semi-Markov approach to continuous time random walk limit
processes}
\runtitle{Semi-Markov approach to CTRW limits}

\begin{aug}
\author[a]{\fnms{Mark M.} \snm{Meerschaert}\thanksref{t1,t2}\ead[label=e1]{mcubed@stt.msu.edu}}
\and
\author[b]{\fnms{Peter} \snm{Straka}\corref{}\thanksref{t2}\ead[label=e2]{p.straka@unsw.edu.au}}
\thankstext{t1}{Supported in part by NSF Grants DMS-10-25486 and DMS-08-03360.}
\thankstext{t2}{Supported in part by NIH Grant R01-EB012079.}
\runauthor{M.~M. Meerschaert and P.~Straka}

\affiliation{Michigan State University and UNSW Australia}

\address[a]{Department of Statistics and Probability\\
Michigan State University\\
East Lansing, Michigan 48824\\
USA\\
\printead{e1}}

\address[b]{School of Mathematics and Statistics\\
UNSW Australia\\
Sydney, NSW 2052\\
Australia\\
\printead{e2}}

\end{aug}

\received{\smonth{6} \syear{2012}}
\revised{\smonth{10} \syear{2013}}

%
\begin{abstract}
Continuous time random walks (CTRWs)
are versatile models for anomalous diffusion processes that have found
widespread application in the quantitative sciences. Their scaling
limits are
typically non-Markovian, and
the computation of
their finite-dimensional distributions is an important open problem.
This paper develops a general semi-Markov theory for CTRW limit
processes in $\mathbb{R}^d$ with infinitely many particle jumps
(renewals) in
finite time intervals. The particle jumps and waiting times can be
coupled and vary with space and time.
By augmenting the state space to include the scaling limits of renewal
times, a CTRW limit process can be embedded in a Markov process.
Explicit analytic expressions for the transition kernels of these
Markov processes are then derived, which allow
the computation of all finite dimensional distributions for CTRW
limits. Two examples illustrate the proposed method.
\end{abstract}

%

\pdfkeywords{60K15, 60F17, 60K20, Continuous time random walk, semi-Markov process, functional limit theorem, renewal theory, anomalous diffusion, time-change, Levy process}
\end{frontmatter}

\section{Introduction}
Continuous time random walks (CTRWs) assume a random waiting time
between each successive jump. They are used in physics to model a
variety of anomalous diffusion processes (see Metzler and Klafter \cite
{Metzler2000}), and have found applications in numerous other fields
(see, e.g., \cite{Berkowitz06,Henry2000,Scalas2000,SchumerMIM}).
The scaling limit of the CTRW is a time-changed Markov process in
$\mathbb{R}^d$
\cite{limitCTRW}. The clock process is the hitting time of an
increasing \levy\ process, which is non-Markovian. The distribution of
the scaling limit at one fixed time $t$ is then usually calculated by
solving a fractional Fokker--Planck equation \cite{Metzler2000}, that
is, a governing equation that involves a fractional derivative in time.
The analysis of the joint laws at \emph{multiple times}, however,
becomes much more complicated, since the limit process is not
Markovian. In fact, the joint distribution of the CTRW limit at two or
more different times has yet to be explicitly calculated, even in the
simplest cases; see Baule and Friedrich \cite{Baule2007} for further
discussion.

The main motivation of this paper is to resolve this problem, and our
approach is to develop the semi-Markov theory for CTRW scaling limits.
CTRWs are renewed after every jump. As it turns out, the discrete set
of renewal times of CTRWs converges to a ``regenerative set'' in the
scaling limit, which is not discrete and can be a random fractal or a
random set of positive Lebesgue measure. This regenerative set allows
for the definition of the scaling limit of the previous and next
renewal time after a time $t$. By incorporating these times into the
state space, a CTRW limit can become Markovian. Although CTRW scaling
limits have appeared in many applications throughout the literature, to
our knowledge the renewal property has only been studied for a discrete
CTRW. Moreover, CTRW limits are examples for possibly \emph
{discontinuous} semi-Markov processes with infinitely many renewals in
finite time, and hence the development here complements the literature
on continuous semi-Markov processes \cite{Harlamov}.

It is known \cite{Kaspi1988} that semi-Markov processes can be
constructed by assuming a Markov additive process $(A_u,D_u)$ and
defining $X_t = A(E_t)$, where $E_t$ is the hitting time of the level
$t$ by the process $D_u$. With this procedure, one also constructs CTRW
limit processes.
However, such CTRW limits are homogeneous in time, and several
applications require time-\emph{in}homogeneous CTRW limit processes
\cite{HLS10PRL,Magdziarz2008}. Hence, we will assume that $(A_u,D_u)$
is a diffusion process with jumps (such that $D_u$ is strictly
increasing), modeling the cumulative sum of non-i.i.d. jumps and
waiting times (see Section~\ref{sec:spctim}) which vary with time and
space. In this setting, we develop a semi-Markov theory for time-\emph
{in}homogeneous CTRW limits.

Coupled CTRW limits, for which waiting times and jumps are not
independent, turn out to be particularly interesting. As recently
discovered \cite{StrakaHenry,OCTRW2}, switching the order of waiting
time and jump (i.e., jumps precede waiting times) yields a different
scaling limit called the overshooting CTRW limit (OCTRW limit). The two
processes can have completely different tail behavior \cite{OCTRW2},
and hence provide versatile models for a variety of relaxation
behaviors in statistical physics \cite{Weron10}. Both CTRW and OCTRW
limit turn out to be semi-Markov processes; however, incorporating the
previous renewal time only renders the CTRW limit Markovian and not the
OCTRW limit, and the opposite is true when the following renewal time
is incorporated. In the uncoupled case, CTRW and OCTRW have the same
limit, and hence both approaches yield Markov processes.

This paper gives explicit formulae for the joint transition
probabilities of the CTRW limit (resp., OCTRW limit), together with
its previous renewal time (resp., following renewal time, see
Section~\ref{sec:embedding}). These formulae facilitate the
calculation\vadjust{\goodbreak}
of all finite-dimensional distributions for CTRW (and \mbox{OCTRW}) limits.
The time-homogeneous case is discussed in Section~\ref{sec4}.
Finally, Section~\ref{sec:fdd}
provides some explicit examples, for problems of current interest in
the physics literature.


\section{Random walks in space--time} \label{sec:spctim}

A continuous time random walk (CTRW) is a random walk in space--time,
with positive jumps in time. Let $c > 0$ be a scaling parameter, and let
\[
\bigl(S^c_n,T^c_n\bigr) =
\bigl(A^c_0,D^c_0\bigr) +\sum
_{k=1}^n \bigl(J^c_k,W^c_k
\bigr)
\]
denote a Markov chain on $\mathbb{R}^d\times[0,\infty)$ that tracks
the position
$S^c_n$ of a randomly selected particle after $n$ jumps, and the time
$T^c_n$ the particle arrives at this position. The particle starts at
position $A^c_0$ at time $D^c_0$, $N^c_t = \max\{k\geq0\dvtx T^c_k \le t\}
$ counts the number of jumps by time $t$, and the CTRW
\[
X^c_t = S^c_{N^c_t}
\]
is the particle location at time $t$. The waiting times $W^c_k$ are
assumed positive, and when $t<T_0^c$ we define $N^c_t =0$. %
The process $N^c_t$ is inverse to $T^c_n$, in the sense that
$N^c_{T^c_n} = n$. Often the sequence $(J^c_k,W^c_k)$ is assumed to be
independent and identically distributed, which is the appropriate
statistical physics model for particle motions in a heterogeneous
medium whose properties are invariant over space and time. The
dependence on the time scale $c>0$ facilitates triangular array
convergence schemes, which lead to a variety of interesting limit
processes \mbox{\cite{Baeumer2005,Baeumer2009,clusterCTRW,triCTRW}}. The CTRW
is called \emph{uncoupled} if the waiting time $W^c_k$ is independent of
the jump~$J^c_k$; see, for example, \cite{coupleCTRW}. Coupled CTRW
models have been applied in physics~\cite{Metzler2000,SKW} and finance
\cite{coupleEcon,Scalas2006Lecture}. If the waiting times are i.i.d. and the jump distribution depends on the current position in space and
time, the CTRW limit is a time-changed Markov process governed by a
fractional Fokker--Planck equation~\cite{HLS10PRL}. A~closely related
model called the overshooting CTRW (OCTRW) is
\[
Y^c_t = S^c_{N^c_t+1},
\]
a particle model for which $J^c_1$ is the random initial location, and
each jump $J^c_k$ is \textit{followed} by the waiting time $W^c_k$. See
\cite{JWZ09} for applications of OCTRW in finance, where $Y_t$
represents the price at the next available trading time. See \cite
{Weron10} for an application of OCTRW to relaxation problems in physics.

In statistical physics applications, it is useful to consider the {\it
diffusion limit} of the (O)CTRW as the time scale $c\to\infty$. To make
this mathematically rigorous, let ${\mathbb D}([0,\infty),\mathbb{R}^{d+1})$
denote the space of c\`adl\`ag functions $f\dvtx [0,\infty)\to\mathbb
{R}^{d+1}$ with
the Skorokhod $J_1$ topology, and suppose
%
\begin{equation}
\label{MMMrwConv} \bigl(S^{c}_{[cu]},T^{c}_{[cu]}
\bigr) =\bigl(A^c_0,D^c_0\bigr)
+\sum_{k=1}^{[cu]}\bigl(J^{c}_k,W^{c}_k
\bigr)\Rightarrow(A_u,D_u),
\end{equation}
where ``$\Rightarrow$'' denotes the weak convergence of probability
measures on ${\mathbb D}([0,\infty)$, $\mathbb{R}^{d+1})$ as $c\to
\infty$.
Suppose
the limit process $(A_u,D_u)$ is a \textit{canonical Feller process} with
state space $\mathbb{R}^{d+1}$, in the sense of \cite
{revuz-yor-2004}, III Section~2. That is, we assume a stochastic basis $(\Omega,
\mathcal F_\infty, \mathcal F_u, \pr^{\chi,\tau})$ in which $\Omega
$ is
the set of right-continuous paths in $\mathbb{R}^{d+1}$ with
left-limits and
$(A_u(\omega),D_u(\omega)) = \omega(u)$ for all $\omega\in\Omega$.
The filtration $\mathcal F = \{\mathcal F_u\}_{u \ge0}$ is right
continuous and $(A_u,D_u)$ is $\mathcal F$-adapted.
The laws $\{\pr^{\chi,\tau}\}_{(\chi,\tau)\in\mathbb{R}^{d+1}}$
are determined
by a Feller semigroup of transition operators $(T_u)_{u \ge0}$ and are
such that $(A_0,D_0) = (\chi,\tau)$, $\pr^{\chi,\tau}$-a.s. The
$\sigma
$-fields $\mathcal F_\infty$ and $\mathcal F_0$ are augmented by the
$\pr^{\chi,\tau}$-null sets. Expectation with respect to $\pr^{\chi
,\tau
}$ is denoted by $\E^{\chi,\tau}$. The map $(\chi,\tau) \mapsto\E
^{\chi
,\tau}[Z]$ is Borel-measurable\vadjust{\goodbreak} for every $\mathcal F_\infty$-measurable
random variable $Z$. If the space--time jumps form an infinitesimal
triangular array (\cite{RVbook}, Definition~3.2.1), then
$(A_u,D_u)-(A_0,D_0)$ is a \levy\ process \cite{triCTRW}. In the
uncoupled case, $A_u-A_0$ and $D_u-D_0$ are independent \levy\
processes \cite{limitCTRW}. If the space--time jump distribution depends
on the current position, it was argued in~\cite{HLS10PRL,Weron2008}
that the limiting process $(A_u,D_u)$ is a jump-diffusion in $\mathbb
{R}^{d+1}$.

If \eqref{MMMrwConv} holds,
and if
%
\begin{equation}
\label{assumption-increasing}
\begin{tabular}{p{280pt}@{}}
the sample paths $u \mapsto D_u$ are $
\pr^{\chi,\tau}$-a.s. strictly increasing and unbounded,
\end{tabular}
\end{equation}
then \cite{StrakaHenry}, Theorem~3.6, implies that
%
\begin{eqnarray}
\label{defCTRWL} X^c_t \Rightarrow X_{t}:=(
A_{E_t-})^{+} \quad\mbox{and}\quad Y^c_t
\Rightarrow Y_{t}:=A_{E_t}
\nonumber
\\[-8pt]
\\[-8pt]
 \eqntext{\mbox{in ${\mathbb D}\bigl([0,
\infty),\mathbb{R}^d\bigr)$ as $c\to\infty$,}}
\end{eqnarray}
where
%
\begin{equation}
\label{EtDef} E_t = \inf\{u > 0\dvtx D_u > t\}
\end{equation}
is the first passage time of $D_u$, so that $E_{D_u} = u$.
Then the inverse process \eqref{EtDef} is defined on all of $\R$ and
has a.s. continuous sample paths. The CTRW limit (CTRWL) process $X_t$
in \eqref{defCTRWL} is obtained by evaluating the left-hand limit of
the outer process $A_{u-}$ at the point $u=E_t$, and then modifying
this process to be right-continuous. This changes the value of the
process at time points $t>0$ such that $u=E_t$ is a jump point of the
outer process $A_u$, and $E_{t+\varepsilon}>E_t$ for all $\varepsilon
>0$. If $A_u$ and $D_u$ have no simultaneous jumps, then the CTRW limit
$X_t$ equals the OCTRW limit $Y_t$ (\cite{StrakaHenry}, Lemma~3.9).
However, these two processes can be quite different in the coupled
case. For example, if $J^c_k=W^c_k$ form a triangular array in the
domain of attraction of a stable subordinator $D_u$,
and if $A_0=D_0=0$, then $A_u=D_u$, and $X_t=D_{E_t-}<t<D_{E_t}=Y_t$
almost surely \cite{bertoin}, Theorem III.4. See Example~\ref{ExK} for
more details.

We assume the Feller semigroup $T_u$ that governs the process
$(A_u,D_u)$ acts on the space $C_0(\mathbb{R}^{d+1})$ of continuous real-valued
functions on $\mathbb{R}^{d+1}$ that vanish at~$\infty$, and that it
admits an
infinitesimal generator $\mathcal A$ of jump-diffusion form
\cite{Applebaum}, equation~(6.42).
In light of \eqref{assumption-increasing}, this generator takes the form
%
\begin{eqnarray}
\label{eq:generator} \mathcal Af(x,t) &=& \sum_{i=1}^d
b_i(x,t) \del_{x_i} f(x,t) + \gamma(x,t)
\del_t f(x,t)
\nonumber\\
&&{}+ \frac{1}{2} \sum_{1\le i,j \le d} a_{ij}(x,t)
\del^2_{x_i x_j}f(x,t)
\nonumber
\\[-8pt]
\\[-8pt]
\nonumber
&&{}+ \int \Biggl[ f(x+y,t+w) - f(x,t) \\
&&\hspace*{34pt}{}- \sum_{i=1}^d
h_i(y,w) \del_{x_i} f(x,t) \Biggr] K(x,t;dy,dw),\nonumber
\end{eqnarray}
where $(x,t) \in\mathbb{R}^{d+1}$, $b_i$ and $\gamma$ are real-valued
functions, and $A=(a_{ij})$ is a function taking values in the
nonnegative definite $d\times d$-matrices. Here, $K(x,t;dy,dw)$ is a
jump-kernel from $\mathbb{R}^{d+1}$ to itself, so that for every
$(x,t) \in
\mathbb{R}^{d+1}$,
$C \mapsto K(x,t;C)$ is a measure on $\mathbb{R}^{d+1}$ that is finite
on sets
bounded away from the origin, and $(x,t)\mapsto K(x,t;C)$ is a
measurable function for every Borel set $C \subset\mathbb{R}^{d+1}$.
The truncation function
$h_i(x,t) = x_i \1\{(x,t) \in[-1,1]^{d+1}\}$.
Since the sample paths of $D_u$ are strictly increasing, $\gamma\ge
0$, the diffusive component of $D_u$ is zero, and the measures
$K(x,t;dy,dw)$ are supported on
$(dy,dw)\in\mathbb{R}^d\times[0,\infty)$.
Instead of assuming that $K(x,t;dy,dw)$ integrates
$1\wedge\|(y,w)\|^2$, it then suffices to assume
%
\begin{equation}
\int \bigl[1 \wedge\bigl(\|y\|^2 + |w|\bigr) \bigr] K(x,t;dy,dw) <
\infty \qquad\forall (x,t) \in\mathbb{R}^{d+1}.
\end{equation}
The space--time jump kernel $K$ can be interpreted as the joint
intensity measure for the long jumps and long waiting times which do
not rescale to $0$ as $c \to\infty$.
If the measures $(dy,dw) \mapsto K(x,t;dy,dw)$ are supported on ``the
coordinate axes'' $(\mathbb{R}^d\times\{0\}) \cup(\{0\} \times
[0,\infty))$,
then large jumps occur independently of long waiting times, and the
CTRWL and OCTRWL are identical (\cite{StrakaHenry}, Lemma~3.9). We refer
to this as the \emph{uncoupled} case, and to the opposite case as the
\emph{coupled} case.

Finally, we assume that the coefficients $b_i$, $\gamma$, $a_{ij}$ and
$K$ satisfy Lipschitz and growth conditions as in \cite{Applebaum}, Section~6.2, so that $(A_u,D_u)$ has an interpretation as
the solution to a stochastic differential equation, as well as a
semimartingale \cite{JacodShiryaev}, Section III.2.
Then for any canonical Feller process $(A_u,D_u)$ on $\mathbb
{R}^{d+1}$, we
define the CTRWL process $X_{t} =(A_{E_t-})^{+}$, and the OCTRWL
process $Y_t = A_{E_t}$, where $E_t$ is given by \eqref{EtDef}. If we
set $A_{0-} = A_0$, then $E_t$, $X_t$ and $Y_t$ are defined for all $t
\in\R$.

\subsection{Forward and backward renewal times}
Although the (O)CTRWL is not Markovian, it turns out that it can be
embedded in a Markov process on a higher dimensional state space, by
incorporating information on the forward/backward renewal times. Define
the \textit{regenerative set}
\[
\mathbf M = \bigl\{(t,\omega) \subset\R\times\Omega\dvtx t = D_u(
\omega) \mbox { for some }u \ge0\bigr\},
\]
the random set of image points of $D_u$. These will turn out to be the
renewal points of the inverse process $E_t$ defined in \eqref{EtDef}.
Since $D_u$ is c\`adl\`ag and has a.s. increasing sample paths, for
almost all $\omega$ the complement of the $\omega$-slice
$\mathbf M(\omega):= \{t \in\R\dvtx (t,\omega) \in\mathbf M\}$ in $\R$
is a countable union of intervals of the form
$[D_{u-}(\omega),D_u(\omega))$, where $u \ge0$ ranges over the jump
epochs of the process $D_u$.
For example, if $D_u$ is compound Poisson with positive drift, then
$\mathbf M$ is a.s. a union of intervals $[a,b)$ of positive length.
If $D_u$ is a $\beta$-stable subordinator with no drift, then $\mathbf
M$ is a.s. a fractal of dimension $\beta$ \cite{Bertoin04}.

For any $t\geq0$, we write $G_t$, the last time of regeneration before
$t$, and $H_t$, the next time of regeneration after $t$, as
%
\begin{equation}\quad
\label{GHdef} G_t(\omega):= \sup\bigl\{s \le t\dvtx s \in\mathbf M(
\omega)\bigr\} \le t \le \inf\bigl\{s > t\dvtx t \in\mathbf M(\omega)\bigr\} =:
H_t(\omega),
\end{equation}
where for convenience we set
$G_t(\omega) = \inf\mathbf M(\omega) = \tau$, $\pr^{\chi,\tau
}$-a.s. whenever the supremum is taken over the empty set. In terms of the CTRW
model, the particle has been resting at its current location since time
$G_t$, and will become mobile again at time $H_t$.
It will become clear in the sequel that the future evolution of $X_t$
and $Y_t$ on the time interval $[H_t, \infty)$ depends only the
position $Y_t$ at time $t = H_t$, meaning that $H_t$ is a \textit{Markov
time} for $X_t$ and $Y_t$.

Note that $G_t$ and $H_t$ are a.s. defined for all $t \in\R$ and
their sample paths are c\`adl\`ag.
By our assumptions on $D_u$ and the definition \eqref{EtDef}, it is
easy to see that
\begin{eqnarray*}
G_{t-} &= D_{E_t-} \quad\mbox{and}\quad H_t =
D_{E_t}, \qquad\pr^{\chi,\tau
}\mbox{-a.s.}
\end{eqnarray*}
The \textit{age process} $V_t$ and the \textit{remaining lifetime}
$R_t$ from
renewal theory can be defined by
%
\begin{equation}
\label{VRdef} V_t:= t-G_t \qquad\mbox{and}\qquad R_t
:= H_t - t \qquad\mbox{for all } t \in\R.
\end{equation}
At any time $t>0$, the particle has been resting at its current
location for an interval of time of length $V_t$, and will move again
after an additional time interval of length $R_t$. We will show below
that the processes $(X_{t-},V_{t-})$ and $(Y_t,R_t)$ are Markov, and we
will compute the joint distribution of these $\mathbb{R}^{d+1}$-valued
processes
at multiple time points, using the Chapman--Kolmogorov equations.
The joint laws of $(X_{t-}, Y_t, V_{t-}, R_t)$ were first calculated in
\cite{Cinlar1976,Maisonneuve77}, but only in the case where the
space--time process $(A_u,D_u)$ is Markov additive (see Section~\ref{sec4}) and only for Lebesgue-almost all $t \ge0$.
We now calculate this joint law in our more general time-inhomogeneous
setting, for \emph{all} $t \ge0$. We need the following additional
definitions: Let
\[
\mathbf C = \bigl\{(t,\omega) \subset\R\times\Omega\dvtx D_{u-}(
\omega) = t = D_u(\omega) \mbox{ for some }u > 0\bigr\}\subset\mathbf
M
\]
be the random set
of points traversed \emph{continuously} by $D_u$. The set $\mathbf C $
is obtained by removing from the set $\mathbf M$ of regenerative points
all points $t$ which satisfy $t=D_u>D_{u-}$ for some $u>0$ (i.e., the
right end points of all contiguous intervals).
Moreover, since $(A_u,D_u)$ visits each point in $\mathbb{R}^{d+1}$ at most
once, it admits a $0$-potential, or mean occupation measure, $U^{\chi
,\tau}$ { defined via}
\begin{eqnarray*}
\int f(x,t) U^{\chi,\tau}(dx,dt) &=& \E^{\chi,\tau} \biggl[\int
_0^\infty f(A_u,D_u) \,du
\biggr] = \int_0^\infty T_uf(\chi,
\tau)\,du
\\
&=& \E^{\chi,\tau} \biggl[\int_0^\infty
f(A_{u-},D_{u-}) \,du \biggr]
\end{eqnarray*}
for any nonnegative measurable function
$f\dvtx \mathbb{R}^{d+1}\to[0,\infty)$. The last equality holds because
$(A_u,D_u)$
only jumps countably many times.
Since $(A_u,D_u)$ has infinite lifetime, $U^{\chi,\tau}$ is an
infinite measure.
We assume that $D_u$ is transient \cite{chung2005markov}, so that
$U^{\chi,\tau}(\mathbb{R}^d\times I) < \infty$ for any compact
interval $I
\subset[0,\infty)$. For instance, any subordinator is transient \cite
{Bertoin04}.

Next we derive the joint law of the Markov process
$(X_{t-},Y_t,V_{t-},R_t)$. The proof uses sample path arguments, and we
consider two cases, starting with the case $\{t \notin\mathbf C\}$:

\begin{proposition}
\label{propNotInC}
Fix $(\chi,\tau) \in\mathbb{R}^{d+1}$ and $t \ge\tau$.
Then
%
\begin{eqnarray}
\label{eqNotInC}
&&\E^{\chi,\tau} \bigl[f(X_{t-},Y_t,V_{t-},R_t)
\1\{t\notin\mathbf C\} \bigr]
\nonumber\\
&&\qquad= \int_{x\in\mathbb{R}^d} \int_{s \in[\tau,t]}
U^{\chi,\tau
}(dx, ds)
\\
&&\qquad\quad{}\times\int_{y \in\mathbb{R}^d} \int_{w \in[t-s,\infty)} K
(x,s;dy,dw) f\bigl(x,x+y,t-s,w-(t-s)\bigr)\nonumber 
\end{eqnarray}
for all nonnegative measurable $f$ defined on $\mathbb{R}^{d+1}\times
\mathbb{R}^{d+1}$.
\end{proposition}

\begin{pf}
The complement of the section set $\mathbf C(\omega)$ in $\R$ is
a.s. a countable union of closed intervals $[D_{u-},D_u]$, where $u$ is a
jump epoch of $D_u$. Hence, for $t\notin\mathbf C$ we have
$G_{t-}\leq t\leq H_t$ and $G_{t-} < H_t$, hence
$\Delta D_{E_t} = H_t - G_{t-}>0$.
In the complementary case $\{t \in\mathbf C\}$, the sample path of
$E_t$ is left-increasing at $t$, and hence the $\mathcal F$-optional
time $E_t$ is announced by the optional times $E_{t-1/n}$.
Hence,
$ E_t^*:= E_t \cdot\1\{t \in\mathbf C\} + \infty\cdot\1\{t \notin
\mathbf C\} $
is $\mathcal F$-predictable (\cite{Kallenberg}, page 410), and since in our
setting $(A_u,D_u)$ is a canonical Feller process, it is quasi
left-continuous (\cite{Kallenberg}, Proposition~22.20), and $\Delta(A,D)_{E_t}
= (0,0)$ a.s.
Writing\vadjust{\goodbreak} $\mathcal J = \{(u,\omega)\in\R^+ \times\Omega\dvtx  \Delta(A,D)_u
\neq(0,0)\}$ for the random set of jump epochs of $(A_u,D_u)$, we
hence find that
\begin{eqnarray*}
&&f(X_{t-},Y_t,V_{t-},R_t) \1\{t
\notin\mathbf C\}
\\
&&\qquad=\sum_{u \in\mathcal J} f(A_{u-},A_u,t-D_{u-},D_u-t)
\1\{D_{u-}\le t \le D_u\}
\\
&&\qquad=\sum_{u \in\mathcal J} f(A_{u-},A_{u-}+
\Delta A_u,t-D_{u-},D_{u-}+ \Delta
D_u-t)
\\
&&\hspace*{17pt}\qquad\quad{}\times\1\{D_{u-}\le t \le D_{u-} + \Delta
D_u\},
\end{eqnarray*}
noting that all members of the sum except exactly one ($u = E_t$) equal $0$.
The last expression equals $\int W(\omega,u;x,s) \mu(\omega,du;dy,dw)$
for the optional random measure
%
\begin{equation}
\label{randomMeasure} \mu(\omega,du;dy,dw)=\sum_{v\ge0}
\1_{\mathcal{J}}(v,\omega) \delta_{(v,\Delta(A_v(\omega),D_v(\omega))}(du;dy,dw)
\end{equation}
on $du\times(dy,dw)\in\R^+\times\mathbb{R}^{d+1}$ associated with
the jumps of
$(A_u,D_u)$, and the predictable integrand
\begin{eqnarray*}
W(\omega,u;y,w) &:= &f\bigl(A_{u-}(\omega),A_{u-}(
\omega)+y,t-D_{u-}(\omega),D_{u-}(\omega )+w-t\bigr)
\\
&&{}\times\1\bigl\{D_{u-}(\omega)\le t\le D_{u-}(\omega)
+ w\bigr\}.
\end{eqnarray*}
The compensator $\mu^p$ of $\mu$ equals \cite{JacodShiryaev}, page 155
%
\begin{equation}
\label{compensator} \mu^p(\omega; du, dy, dw) = K\bigl(A_{u-}(
\omega),D_{u-}(\omega);dy, dw\bigr) \,du.
\end{equation}
Then the compensation formula \cite{JacodShiryaev}, II.1.8, implies that
\begin{eqnarray*}
&&\E^{\chi,\tau} \bigl[f(X_{t-},Y_t,V_{t-},R_t)
\1\{t\notin\mathbf C\} \bigr]
\\
&&\qquad=\E^{\chi,\tau} \biggl[\int W(\omega,u;y,w) \mu^p(\omega
,du;dy,dw) \biggr]
\\
&&\qquad=\E^{\chi,\tau} \biggl[\int_{u=0}^\infty\int
_{y\in\mathbb{R}^d} \int_{w =
0}^\infty f
\bigl(A_{u-}(\omega),A_{u-}(\omega)+y,t-D_{u-}(
\omega),\\
&&\hspace*{229pt}{}D_{u-}(\omega )+w-t\bigr)
\\
&&\hspace*{98pt}\qquad\quad{}\times \1\bigl\{D_{u-}(\omega)\le t\le D_{u-}(\omega)
+ w\bigr\}\\
&&\hspace*{150pt}{}\times K\bigl(A_{u-}(\omega),D_{u-}(\omega);dy, dw\bigr)
\,du \biggr]
\\
&&\qquad=\int_{x\in\mathbb{R}^d} \int_{s=\tau}^\infty
\int_{y\in\mathbb{R}^d} \int_{w = 0}^\infty
f(x,x+y,t-s,s+w-t)
\\
&&\hspace*{99pt}\qquad\quad{}\times\1\{s\le t\le s+w\} K(x,s;dy, dw) U^{\chi,\tau}(dx, ds),
\end{eqnarray*}
which is equivalent to \eqref{eqNotInC}.
\end{pf}

The following proposition handles the case $\{t \in\mathbf C\}$.
%
\begin{proposition} \label{prop:C}
Fix $(\chi,\tau) \in\mathbb{R}^{d+1}$ and $t\ge\tau$.
Suppose that the temporal drift $\gamma$ is bounded and continuous, and
assume that the mean occupation measure $U^{\chi,\tau}(dx,dt)$ is
Lebesgue-absolutely continuous with a continuous density
$u^{\chi,\tau}(x,t)$.
Then
%
\begin{equation}
\label{eq:propC} \E^{\chi,\tau} \bigl[f(Y_{t})\1\{t\in\mathbf{C}\}
\bigr] =\int_{x \in\mathbb{R}^d}f(x)\gamma(x,t)u^{\chi,\tau}(x,t) \,dx
\end{equation}
for all bounded measurable $f$.
Also \eqref{eq:propC} remains true if $Y_{t}$ is replaced by $X_{t-}$,
$Y_{t-}$ or $X_{t}$.
\end{proposition}

\begin{pf}
Similarly to the proof in \cite{Maisonneuve77},
$D_u$ admits a decomposition into a continuous and a discontinuous part via
\[
D_{u}^{c} =\int_{0}^{u}
\gamma(A_s,D_s)\,ds,\qquad
D_{u}^{d} =\sum_{0\le s\le u}\Delta
D_{s},\qquad t \ge0.
\]
To see this, we first note that $(A_u,D_u)$ is a semimartingale, %
and hence $D_u$ allows the decomposition
%
\begin{equation}
\label{decomp} D_u = \sum_{s\le u} \Delta
D_s \1\{\Delta D_s > 1\} + B_u +
M_u,
\end{equation}
where $B_u$ is a predictable process of finite variation (the first
characteristic of $D_u$) and $M_u$ is a local martingale.
Due to \cite{JacodShiryaev}, IX Section~4a, and \eqref{eq:generator},
$B_u = \int_0^u \tilde\gamma(A_s,D_s)\,ds$ where $\tilde\gamma(x,t) =
\gamma(x,t) + \int s \1\{\|(y,s)\|\le1\} K(x,t;dy,ds)$.\break
Since $D_u$ has no diffusive part, $M_u$ is purely discontinuous and
equal to
$M_u = \sum_{s \le u}\Delta D_s\1\{\Delta D_s \le1\} - \int_0^u\int
w\1
\{\|(y,w)\|\le1\} K(A_s,D_s;dy,dw)\,ds$.\break  But then~\eqref{decomp} reads
$D_u = D_u^d + D_u^c$.

For fixed $\omega$, the paths of $D$, $D^c$ and $D^d$ are nondecreasing
and define Lebesgue--Stieltjes measures $dD$, $dD^c$ and $dD^d$ on
$[0,\infty)$.
Then for any bounded measurable $f$ and $g$, we have
%
\begin{equation}
\label{IntegralLeft} \int_{0}^{\infty}f(A_{u})g(D_{u})
\gamma(A_{u},D_{u})\,du =\int_{0}^{\infty}f(A_{u})g(D_{u})\,dD_{u}^{c}.
\end{equation}
The continuous measure $dD^{c}$ does not charge the countable set $\{
u\dvtx\break \Delta D_{u}\neq0\}$ of discontinuities of $D_u$ and coincides with
$dD$ on the complement $\{u\dvtx  \Delta D_{u}=0\}$.
Hence the right-hand side of \eqref{IntegralLeft} can be written as
%
\begin{equation}
\label{IntegralLeft2} \int_{0}^{\infty}f(A_{u})g(D_{u})
\1\{u\dvtx \Delta D_{u}=0\}\,dD_{u}.
\end{equation}
The following substitution formula holds for all right-continuous,
unbounded and strictly increasing $F \dvtx [0,\infty) \to[0,\infty)$, the
inverse $F^{-1}(t) = \inf\{u\dvtx\break F(u) > t\}$ and measurable $h\dvtx [0,\infty)
\to[0,\infty)$:
\[
\int_0^\infty h(u) \,dF_u = \int
_0^\infty h\bigl(F^{-1}(t)\bigr) \,dt.
\]
To see this, first show the statement for $h$ an indicator function of
an interval $(a,b] \subset[0,\infty)$ and then for a function taking
finitely many values. The statement for positive $h$ then follows by
approximation via a sequence of finitely valued functions from below,
and for general $h$ by a decomposition into positive and negative part.
Applying the substitution formula to \eqref{IntegralLeft2} with $F(u) =
D_u$, the right-hand side of \eqref{IntegralLeft} reduces to
\[
\int_{0}^{\infty}f(Y_{t})g(H_t)
\1\{t\dvtx \Delta D_{E_t}=0\}\,dt.
\]
Now note that $\Delta D_{E_t}=0$ is equivalent to $t\in\mathbf{C}$ and
implies $H_t=t$.
Hence, the above lines show that the left-hand side of \eqref
{IntegralLeft} equals
\[
\int_{0}^{\infty}f(Y_{t})g(t)\1\{t\in
\mathbf{C}\}\,dt.
\]
Take expectations and apply Tonelli's theorem to get
\[
\int_{\mathbb{R}^{d+1}}f(x)\gamma(x,t)g(t)u^{\chi,\tau}(x,t) \,dx \,dt =
\int_{0}^{\infty}\E^{\chi,\tau}
\bigl[f(Y_{t})\1\{t\in\mathbf{C}\}\bigr]g(t)\,dt.
\]
Since $g$ is an arbitrary nonnegative bounded measurable function,
this yields \eqref{eq:propC} for almost every $t$.
{ By our assumption that $D_u$ is transient, $U^{\chi,\tau}(\mathbb
{R}^d\times I)
< \infty$ for compact $I \subset[0,\infty)$, and then it can be seen
that the continuous function $u^{\chi,\tau}(x,t)$ must be bounded on
$\mathbb{R}^d\times I$.
Let $I$ contain $t$ and apply dominated convergence to see that}
the right-hand side of \eqref{eq:propC} is continuous in $t$.
We have already noted in the proof of Proposition~\ref{propNotInC} that
$\Delta(A,D)_{E_t} = (0,0)$ on $\{t \in\mathbf C\}$, which shows the
continuity of the left-hand side.
This shows the equality for \emph{all} $t \ge0$, and also that
$X_t - X_{t-} = 0 = Y_t - Y_{t-}$ on $\{t \in\mathbf C\}$.
\end{pf}

We can now characterize the joint law of $(X_{t-},Y_t,V_{t-},R_t)$:

\begin{theorem} \label{thJointLaw}
Fix $(\chi,\tau) \in\mathbb{R}^{d+1}$ and $t \ge\tau$.
If $\gamma$ does not vanish, then suppose that the mean occupation
measure $U^{\chi,\tau}(dx,dt)$ has a continuous Lebesgue density
$u^{\chi,\tau}(x,t)$, and if
$\gamma\equiv0$, let $u^{\chi,\tau}(x,t) \equiv0$.
Then
\begin{eqnarray*}
&&\E^{\chi,\tau}  \bigl[f(X_{t-},Y_{t},V_{t-},R_t)
\bigr]\\
&&\qquad=\int_{x\in\mathbb{R}^d}f(x,x,0,0)\gamma(x,t)u^{\chi,\tau}(x,t)
\,dx
\\
&&\qquad\quad{}+ \int_{x\in\mathbb{R}^d} \int_{s \in[\tau,t]}
U^{\chi,\tau
}(dx, ds)
\\
&&\quad\qquad{}\times \int_{y\in\mathbb{R}^d} \int_{w \in[t-s,\infty)} K (x,s;dy,dw) f
\bigl(x,x+y,t-s,w-(t-s)\bigr)
\end{eqnarray*}
for all bounded measurable $f$.
Moreover, $X_\tau= Y_\tau= \chi$ and $V_\tau= R_\tau= 0$, $\pr
^{\chi
,\tau}$-almost surely.
\end{theorem}

\begin{pf}
On the set $\{t \in\mathbf C\}$, $V_{t-} = 0 = R_t$. The above formula
then follows from Propositions \ref{prop:C} and \ref{propNotInC}.
Assumption \eqref{assumption-increasing} and the right-continuity of
$D$ yields $V_\tau= R_\tau= 0$. The sample paths of $E$ are then seen
to be right-increasing at $\tau$ and $E_t > E_\tau= 0$ for $t > \tau$.
The right-continuity of $A$ together with $A_0 = \chi$, $\pr^{\chi
,\tau
}$-a.s. yields $X_\tau= Y_\tau= \chi$.
\end{pf}

\section{The Markov embedding} \label{sec:embedding}

In this section, we establish the Markov property of the processes
$(Y_t,R_t)$ and $(X_{t-},V_{t-})$.
Since $\{E_t \le u\} = \{D_u \ge t\}$, $\pr^{\chi,\tau}$-a.s. for
every $(\chi,\tau)\in\mathbb{R}^{d+1}$ \cite{limitCTRW}, equation
(3.2), we
see that
$E_t$ is an $\mathcal F$-optional time for every $t$. We introduce the
filtration $\mathcal H = \{\mathcal H_t\}_{t \in\R}$ where $\mathcal
H_t = \mathcal F_{E_t}$ and note that $(Y_t,R_t)$ is adapted to
$\mathcal H$.
Moreover, if $T$ is $\mathcal H$-optional, then
$E_T\dvtx \omega\mapsto E_{T(\omega)}(\omega)$
is $\mathcal F$-optional (see Lemma~\ref{lemmaOptional}).
We define the family of operators $\{Q_{s,t}\}_{s\le t}$ acting on the
space $B_b(\mathbb{R}^d\times[0,\infty))$ of real-valued bounded measurable
functions $f$ defined on $\mathbb{R}^d\times[0,\infty)$ as follows:
%
\begin{eqnarray}
\label{evolutionQ} 
Q_{s,t}f(y,0) &= &\E^{y,s}
\bigl[f(Y_t,R_t)\bigr],
\nonumber
\\
Q_{s,t}f(y,r) &= &\1\{r > t-s\} f\bigl(y,r-(t-s)\bigr)
\\
&&{}+ \1\{0 \le r \le t-s\} Q_{s+r,t}f(y,0).\nonumber 
\end{eqnarray}
%
The dynamics of $Q_{s,t}$ can be interpreted as follows:
If the process $(Y_t,R_t)$ starts at $(y,r)$, the position in space $y$
does not change while the remaining lifetime $R_t$ decreases linearly
to $0$. When $r = 0$, the process continues with the dynamics given by
$(Y_t,R_t)$ started at location $y$ at time $s+r$.
Note that $Q_{s,t}f(y,r)$ is measurable in $(s,t,y,r)$, for every
bounded measurable $f$, by the construction of the probability measures
$\pr^{\chi,\tau}$.
We can now state the strong Markov property of $(Y_t,R_t)$ with respect
to $\mathcal H$ and $Q_{s,t}$.

\begin{theorem} \label{th:RYembed}
Suppose that the operators $Q_{s,t}$ are given by \eqref{evolutionQ}. Then:

\begin{longlist}[(iii)]
\item[(i)]
The operators $Q_{s,t}$ satisfy the Chapman--Kolmogorov equations:
\[
Q_{q,s}Q_{s,t}f = Q_{q,t} f,\qquad q\le s\le t,
\]
and moreover, $Q_{s,t}\1 = \1$.
\item[(ii)]
Let $(\chi,\tau) \in\mathbb{R}^{d+1}$, $t \ge0$ and let $T$ be a
$\mathcal
H$-optional time. Then
\[
\E^{\chi,\tau}\bigl[f(Y_{T+t},R_{T+t})|\mathcal
H_T\bigr] = Q_{T,T+t}f(Y_T,R_T),
\qquad\pr^{\chi,\tau}\mbox{-almost surely}
\]
for every real-valued bounded measurable $f$.
\item[(iii)]
The process $t \mapsto(Y_t,R_t)$ is quasi-left-continuous with respect
to $\mathcal H$.
\end{longlist}
Hence, $(Y_t,R_t)$ is a Hunt process with respect to $\mathcal H$ and
transition operators~$Q_{s,t}$.
\end{theorem}

\begin{pf} A proof is given in the \hyperref[app]{Appendix}.
\end{pf}
%

We define the filtration $\mathcal G = \{\mathcal G_t\}_{t \in\R}$ via
$\mathcal G_t = \mathcal F_{E_t-}$, the $\sigma$-field of all
$\mathcal
F$-events strictly before $E_t$. Evidently, the left-continuous process
$(X_{t-},V_{t-})$ is adapted to $\mathcal G$.
The main idea behind the Markov property of $(X_{t-},V_{t-})$ is that,
knowing the current state $(x,v)=(X_{t-},V_{t-})$ and the joint
distribution of the next space--time increment given by the kernel
$K(x,v;dy,dw)$ in \eqref{eq:generator}, one can calculate the
distribution of the next renewal time $H_t$ and the position $Y_t$ at
that time. Then the probability of events after the renewal point $H_t$
can be calculated starting at the point $(Y_t,H_t)$ in space--time.
We introduce the following notation:
Define the family of probability kernels $\{ K_v\}_{v\ge0}$ on
$\mathbb{R}^{d+1}$
%
\begin{eqnarray}
\label{eq:K_a} 
K_v(x,t;C) &=& \frac{K (x,t;C\cap (\mathbb{R}^d\times
[v,\infty)
)  )} {
K (x,t;\mathbb{R}^d\times[v,\infty)  )},\nonumber\\
\\[-20pt]
 \eqntext{v > 0,
(x,t) \in\mathbb{R}^{d+1}, C\subset\mathbb{R}^{d+1},}
\\
K_0(x,t;C) &=& \delta_{(0,0)}(C),\nonumber 
\end{eqnarray}
where $C$ is a Borel set. For $v>0$, $K_v(x,t;dy,dw)$ is the
conditional probability distribution of a space--time jump $(y,w)$ (a
jump-waiting time pair), given that a time-jump (a waiting time)
greater than or equal to $v$ occurs.
Should the denominator $ K(x,t;\mathbb{R}^d\times[v,\infty))$ equal
$0$, we
set $K_v(x,t;C) =0$. If $v=0$, then $K_0$ is the Dirac-measure
concentrated at $(0,0)\in\mathbb{R}^{d+1}$.
Since $v \mapsto K(x,t; C$ $\cap\mathbb\mathbb{R}^d\times[v,\infty
))$ is
decreasing, and hence measurable, it follows
that $v\mapsto K_v(x,t;C)$ is measurable for every $(x,t) \in\mathbb{R}^{d+1}$
and Borel $C\subset\mathbb{R}^{d+1}$.

We now define the family of operators $\{P_{s,t}\}_{s\le t}$ acting on
the space $B_b(\mathbb{R}^d\times[0,\infty))$ of real-valued bounded
measurable
functions defined on $\mathbb{R}^d\times[0,\infty)$:
%
\begin{eqnarray}
\label{evolutionP} 
&&P_{s,t}f(x,0) = \E^{x,s}
\bigl[f(X_{t-},V_{t-})\bigr],\nonumber
\\
\qquad &&P_{s,t}f(x,v)
\nonumber
\\[-8pt]
\\[-8pt]
\nonumber
&&\qquad= f(x,v+t-s) K_v\bigl(x,s-v;
\mathbb{R}^d\times[v+t-s,\infty)\bigr)
\\
&&\qquad\quad{}+\int_{y \in\mathbb{R}^d} \int_{w\in[v,v+t-s)}
P_{s+w-v,t}f(x+y,0) K_v(x,s-v; dy, dw).\nonumber 
\end{eqnarray}
%
The dynamics given by $P_{s,t}$ can be interpreted as follows. With
probability $K_v(x,s-v;\mathbb{R}^d\times(v+t-s,\infty))$, the
process remains
at $x$ and the age increases by $t-s$. This is the probability that the
size of a jump of $D$ whose base point is at $(x,s-v)$ exceeds $v+t-s$,
given that it exceeds $v$. The remaining probability mass for the jump
of $(A,D)$ is spread on the set $(y,w) \in\mathbb{R}^d\times
[v,v+t-s)$, and
the starting point is updated from $x$ to $x+y$ at the time $s-v+w$.

\begin{theorem} \label{th:AXembed}
Let $P_{s,t}$ be the operators defined by \eqref{evolutionP}. Then:
\begin{longlist}[(ii)]
\item[(i)]
The operators $(P_{s,t})$ satisfy the Chapman--Kolmogorov property:
\[
P_{q,s}P_{s,t}f = P_{q,t}f,
\qquad
 q \le s \le t,
\]
and moreover, $P_{s,t}\1 = \1$.
\item[(ii)]
The process $(X_{t-},V_{t-})$ satisfies the simple Markov property with
respect to $\mathcal G$ and $P_{s,t}$:
\[
\E^{\chi,\tau}\bigl[f(X_{t-},V_{t-}) | \mathcal
G_s\bigr] = P_{s,t}f(X_{s-},V_{s-}),\qquad
 \pr^{\chi,\tau}\mbox{-a.s.}
\]
for all $(\chi,\tau) \in\mathbb{R}^{d+1}$, $\tau\le s \le t$ and
real-valued
bounded measurable $f$.
\end{longlist}
\end{theorem}

\begin{pf} A proof is given in the \hyperref[app]{Appendix}.
\end{pf}

\begin{remark}
It would be interesting to investigate whether the \textit{moderate}
Markov property (e.g., see Chung and Glover \cite{Chung1979}) holds for
$(X_{t-},V_{t-})$.
An application of the compensation formula to the process
$(X_{t-},G_{t-})$ might yield a proof, but this would require the
semimartingale characteristics of $(X_{t-},G_{t-})$, which we have not
been able to calculate.
\end{remark}

\section{The time-homogeneous case}\label{sec4}

If the coefficients $b(x,t), \gamma(x,t), a(x,t)$ and $K(x,t;dy,dw)$ of
the generator $\mathcal A$ in \eqref{eq:generator} do not depend on $t
\in\R$, then we say that $(A_u,D_u)$ is a \emph{Markov additive
process}. This means that the future of $(A_u,D_u)$ only depends on the
current state of $A_u$; see, for example,~\cite{Cinlar1972}.

\begin{theorem}\label{TimeHomoTh}
If the space--time random walk limit process $(A_u,D_u)$ in \eqref
{MMMrwConv} is Markov additive, then the Markov processes
$(X_{t-},V_{t-})$ and $(Y_t,R_t)$ are time-homogeneous. Writing
$K_r(x;dy,ds):=K_r(x,t;dy,ds)$
and $\pr^x = \pr^{x,0}$, the transition semigroup $Q_{t-s}:=Q_{s,t}$ of
the Markov process $(Y_t,R_t)$ is given by
%
\begin{eqnarray}
\label{eq:semigroupQ} 
Q_t f(y,0) &= &\E^{y}
\bigl[f(Y_t,R_t)\bigr],
\nonumber
\\[-8pt]
\\[-8pt]
\nonumber
Q_t f(y,r) &=& \1\{0\le t < r\} f(y,r-t) + \1\{0 \le r \le t\}
Q_{t-r}f(y,0) 
\end{eqnarray}
%
and the transition semigroup $P_{t-s}:=P_{s,t}$ of the Markov process
$(X_{t-},V_{t-})$ is given by
%
\begin{eqnarray}
\label{eq:semigroupP} 
P_t f(x,0) &=& \E^x
\bigl[f(X_{t-},V_{t-})\bigr],\nonumber
\\
P_t f(x,v) &=& f(x,v+t) K_v\bigl(x;
\mathbb{R}^d\times[{v+t},\infty)\bigr)
\\
&&{}+\int_{\mathbb{R}^d\times[v,v+t)} P_{v+t-w}f(x+y,0) K_v(x;
dw, dy),\nonumber 
\end{eqnarray}
%
acting on the bounded measurable functions defined on $[0,\infty)
\times\mathbb{R}^d$.
\end{theorem}

\begin{pf}
Since $(A_u,D_u)$ is Markov additive, we have $\vartheta_s \mathcal Af
= \mathcal A \vartheta_s f$ for all $s \in\R$, where the shift
operator $\vartheta_s f(x,t) = f(x,t+s)$. It follows that the
resolvents $(\lambda- \mathcal A)^{-1}$, the semigroup $T_u$ and the
kernel $Uf(\chi,\tau) = U^{\chi,\tau}(f)$ commute with $\vartheta
_s$. Then
$\E^{\chi,\tau}[f(A_u,D_u)] = \E^{\chi,0}[f(A_u,\tau+D_u)]$ for
all $u$
and measurable $f$, and hence it suffices to work with the laws $\pr
^{\chi,0}$.
Now in Theorem~\ref{thJointLaw},writing $U^{\chi,\tau}(dx,dt) =
U^{\chi
}(dx,dt-\tau)$, we have
\[
\E^{\chi,\tau}\bigl[f(X_{t-},Y_t,V_{t-},R_t)
\bigr] = \E^{\chi,0}\bigl[f(X_{(t-\tau)-},Y_{t-\tau},V_{(t-\tau)-},R_{t-\tau})\bigr],
\]
where $\tau= 0$ without loss of generality. It follows that \eqref
{eq:semigroupQ} and \eqref{eq:semigroupP} are semigroups acting on the
bounded measurable functions defined on $[0,\infty) \times\mathbb{R}^d$,
compare \cite{Jacod1973}, equations (19) and (31).
\end{pf}

\begin{remark}\label{RemSec3}
Under the assumptions of Theorem~\ref{thJointLaw}, a simple
substitution yields the formulation of $P_t$ and $Q_t$ in terms of
transition probabilities: For $P_t$, we find
\begin{eqnarray}
\label{Ptxrdzdw}
&&P_{t}(x_0,0;dx,dv)\nonumber\\
&&\qquad=\gamma(x,t)u^{x_0}(x,t)
\,dx \delta_0(dv)
\nonumber\\
&&\qquad\quad{}+K \bigl(x_0;\mathbb{R}^d\times[v,\infty)\bigr)
U^{x_0}(dx, t-dv)\1\{0\leq v\leq t\},
\nonumber
\\[-8pt]
\\[-8pt]
\nonumber
&&P_{t}(x_0,v_0;dx,dv)\\
&&\qquad =
\delta_{x_0}(dx) \delta_{v_0+t}(dv) K_{v_0}
\bigl(x_0;\mathbb{R}^d\times[v_0+t,\infty
)\bigr)\nonumber
\\
&&\qquad\quad{}+ \int_{y \in\mathbb{R}^d} \int_{w\in[v_0,v_0+t)}
P_{v_0+t-w}(x_0+y,0;dx,dv) K_{v_0}(x_0;
dy, dw),\nonumber 
\end{eqnarray}
and for $Q_t$ we have\vspace*{2pt}
%
\begin{eqnarray}
&&Q_{t} (y_0,0;dy,dr)\nonumber\\
&&\qquad=\gamma(y,t)
u^{y_0}(y,t) \,dy \delta_0(dr)
\nonumber\\
&&\qquad\quad{}+ \int_{x\in\mathbb{R}^d} \int_{w\in[0,t]}
U^{y_0}(dx,dw) K(x;dy-x,dr+t-w),
\\
&&Q_{t} (y_0,r_0;dy,dr) \nonumber\\
\nonumber&&\qquad= \1\{0<t <
r_0\} \delta_{y_0}(dy)\delta_{r_0-t}(dr)+ \1\{0 \le r_0 \le t\} Q_{t-r_0}(y_0,0;dy,dr).\nonumber
\end{eqnarray}
%
\end{remark}

\section{Finite-dimensional distributions}\label{sec:fdd}

In this section, we provide two examples to illustrate the explicit
computation of finite dimensional distributions for the CTRWL process
$X_t$ and the OCTRWL process $Y_t$.

\begin{example}[(The inverse stable subordinator)]
A very simple CTRW model takes deterministic jumps $J^c_n=c^{-1}$ and
waiting times $W^c_k$ in the domain of attraction of a standard $\beta
$-stable subordinator $\bar D_u$ such that $\E[e^{-s\bar
D_u}]=e^{-us^\beta}$. Setting $(A_0,D_0) = (\chi,\tau)$, \eqref
{MMMrwConv} holds with $(A_u,D_u)=(\chi+u,\tau+\bar D_u)$, where
$\bar
D_u$ is a $\beta$-stable subordinator. Here, the CTRWL and the OCTRWL
coincide, since $A_u$ has no jumps. If $(\chi,\tau)=(0,0)$, then in
\eqref{defCTRWL} we have $X_t=Y_t=E_t$, the inverse $\beta$-stable
subordinator. Now we will compute the joint distributions of this first
passage time process. The joint Laplace transform of these
finite-dimensional distributions was computed by Bingham \cite{Bingham}
but to our knowledge, the distributions themselves have not been
reported in the literature.

The space--time limit $(A_u,D_u)$ is a canonical Feller process on
$\mathbb{R}^{d+1}$ with generator $\mathcal A$ given by \eqref{eq:generator}
with $d=1$, $b_1\equiv1$, $\gamma\equiv0$, $a_{11}\equiv0$, and jump
kernel $K(x,t;dy,dw)=\delta_0(dy)\Phi(w)\,dw$ by \cite{FCbook}, Proposition~3.10, where the L\'evy measure $\Phi(w)\,dw= \beta w^{-\beta
-1}\,dw/\Gamma(1-\beta)$. The stable L\'evy process $\bar D_u$ has a
smooth density $g(t,u)$ so that $\pr^{0,0}(\bar D_u \in dt) = g(t,u)\,dt$
for every $u>0$ by \cite{Jurek1993}, Theorem~4.10.2.
The underlying process $(A_u,D_u)$ is Markov additive, hence
$(X_{t-},V_{t-})$ and $(Y_t,R_t)$ are time-homogeneous Markov
processes. In \cite{StrakaHenry}, Lemma~4.2, it was shown that
$(X_t,V_t)$, has no fixed discontinuities, hence $(X_{t-},V_{t-})$ has
the same law as $(X_t,V_t)$.
One checks that
the $0$-potential of $(A_u,D_u)$ is absolutely continuous with
density\vspace*{2pt}
%
\begin{equation}
\label{uxtex1} u^{\chi,\tau}(x,t)=g(t-\tau,x-\chi)\1\{t>\tau,x>\chi\}.
\end{equation}
Then it follows from (\ref{Ptxrdzdw}) that the transition semigroup of
$(X_{t-},V_{t-})$ is given by\vspace*{2pt}
\begin{eqnarray}
\label{ptchi0xv2} 
&&P_t(x_0,0;dx,dv)\nonumber\\
&&\qquad=g(t-v,x-x_0)
\Phi(v,\infty)\, dx \,dv,
\\
&&P_t(x_0,v_0;dx,dv)\nonumber
\\
&&\qquad= \delta_{x_0}(dx)
\delta_{v_0+t}(dv) \biggl(\frac{v_0+t}{v_0} \biggr)^{-\beta} \1
\{v_0>0\}
\nonumber\\
\nonumber &&\qquad\quad{}+ \biggl( \frac{v_0}{v} \biggr)^\beta \int_{s=v_0}^{v_0+t-v} g
\bigl((t-v)-(s-v_0),x-x_0\bigr)\\
 &&\hspace*{84pt}\qquad\quad{}\times\frac{\beta s^{-1-\beta}\,ds}{\Gamma(1-\beta)}
 \1
\{x>x_0,0<v<t\} \,dx \,dv.\nonumber 
\end{eqnarray}
Hence, for $0<t_1<t_2$, the joint distribution of
$(E_{t_1},V_{t_1},E_{t_2},V_{t_2})$ is
\begin{eqnarray*}
&&\pr^{0,0}(E_{t_1}\in dx, V_{t_1} \in
dv, E_{t_2} \in dy, V_{t_2} \in dw)
\\[1pt]
&&\qquad= P_{t_2-t_1}(x,v;dy,dw)P_{t_1}(0,0;dx,dv)
\\[1pt]
&&\qquad=g(t_1-v,x)\Phi(v,\infty) \1\{x>0,0<v<t_1\}\,dx \,dv
\\[1pt]
&&\qquad\quad{}\times \biggl[\delta_x(dy)\delta_{v+t_2-t_1}(dw) \biggl(
\frac
{v+t_2-t_1}{v} \biggr)^{-\beta}
\\[1pt]
&&\hspace*{16pt}\qquad\quad{}+\int_{s=v}^{v+t_2-t_1-w} g\bigl((t_2-t_1-w)-(s-v),y-x
\bigr)
\\[1pt]
&&\hspace*{80pt}\qquad\quad{}\times\frac{\beta s^{-\beta-1} \,ds}{\Gamma(1-\beta
)} \biggl(\frac{v}w
\biggr)^\beta dy \,dw\, \1\{y>x,0<w<t_2-t_1\}
\biggr] 
\end{eqnarray*}
since $E_0=V_0=0$ for the physical starting point $(0,0)$. Integrating
out the backward renewal times $V_{t_1}$ and $V_{t_2}$, it follows that
the joint distribution of $(E_{t_1},E_{t_2})$ is
%
\begin{eqnarray}
\label{EtFDD} 
&&\pr(E_{t_1}\in dx, E_{t_2} \in
dy)
\nonumber\\[1pt]
&&\qquad= \1\{x>0\}\delta_x(dy) \int_{v=0}^{t_1}g(t_1-v,x)
\frac
{(v+t_2-t_1)^{-\beta}}{\Gamma(1-\beta)} \,dv
\\
&&\qquad\quad{}+\int_{v=0}^{t_1}\int_{w=0}^{t_2-t_1}
\int_{s=v}^{v+t_2-t_1-w} g\bigl((t_2-t_1-w)-(s-v),y-x
\bigr)\,dy\1\{y>x\}\nonumber
\nonumber\\
&&\hspace*{118pt}\qquad\quad{}\times\frac{\beta s^{-\beta-1} \,ds}{\Gamma(1-\beta
)} \biggl(\frac{v}w
\biggr)^\beta dw \,dv.\nonumber 
\end{eqnarray}
\end{example}

\begin{remark}
The joint distribution of $(E_{t_1},E_{t_2})$ can also be computed from
the OCTRW embedding,
but the computation appears to be simpler using the CTRWL embedding.
\end{remark}

\begin{remark}
Baule and Friedrich \cite{Baule2007} compute the Laplace transform of
the joint distribution function $H(x,y,s,t)$ of $x=E_s$ and $y=E_t$ and
show that
\[
(\partial_x+\partial_y)H(x,y,s,t)=-(
\partial_{s}+\partial _{t})^\beta H(x,y,s,t)
\]
on $0<s<t$ and $0<x<y$. Equation \eqref{EtFDD} provides an explicit
solution to this governing equation, which solves an open problem in
\cite{Baule2007}.
The finite dimensional laws of any uncoupled CTRW limit can easily be
calculated from the finite dimensional laws of $E_t$, given the law of
the process $A_u$. This follows from a simple conditioning argument;
see, for example, \cite{limitCTRW}.
\end{remark}

\begin{example}\label{ExK}
Kotulski \cite{Kotulski95} considered a CTRW with jumps equal to the
waiting times $J^c_n=W^c_k$, in the domain of attraction of a standard
$\beta$-stable subordinator $\bar D_u$ such that $\E[e^{-s\bar
D_u}]=e^{-us^\beta}$.
Equation \eqref{MMMrwConv} holds with $(A_u,D_u)=(A_0+\bar
D_u,D_0+\bar
D_u)$. The space--time limit $(A_u,D_u)$ is a canonical Feller process
on $\mathbb{R}^{d+1}$ with generator $\mathcal A$ given by \eqref
{eq:generator}
with $d=1$, $\gamma\equiv0$ and $K(x,t;dy,dw)=\delta_w(dy)\Phi(dw)$,
where $\Phi(dw)=\phi(w)\,dw= \beta w^{-\beta-1}\,dw/\Gamma(1-\beta)$. The
stable L\'evy process $\bar D_u$ has a smooth density $g(t,u)$ so that
$\pr^{0,0}(\bar D_u \in dt) = g(t,u)\,dt$ for every $u>0$. Since the
Markov process $(A_u,D_u)$ is Markov additive, we need only compute the
potential for $\tau=0$:
\begin{equation}
\label{uxtex2a} U^{\chi,0}(dx,dt)= \delta_{\chi+t}(dx) \int
_{u=0}^\infty g(t,u)\,du \,dt.
\end{equation}
Next, one sees that
\begin{equation}
\label{Intex2} \int_{u=0}^\infty g(t,u) \,du=
\frac{t^{\beta-1}}{\Gamma(\beta)}
\end{equation}
by taking Laplace transforms on both sides (also see \cite{FCbook}, Example~2.9).
The $0$-potential hence equals
%
\begin{equation}
\label{uxtex2} U^{\chi,0}(dx,dt)= \delta_{\chi+t}(dx)
\frac{t^{\beta-1}}{\Gamma
(\beta
)} \,dt.
\end{equation}
With $\Phi([v,\infty)) =v^{-\beta}/\Gamma(1-\beta)$, \eqref
{Ptxrdzdw} reads
\begin{eqnarray*}
P_{t}(x_0,0;dx,dv)&=&\frac{v^{-\beta}}{\Gamma(1-\beta)}
\frac
{(t-v)^{\beta-1}}{\Gamma(\beta)} \delta_{x_0+t-v}(dx) \,dv\1\{0< v< t\},
\\
\label{Ptxrdzdwex2}
P_t(x_0,v_0;dx,dv)&=& \delta_{x_0}(dx)
\delta_{v_0+t}(dv) \biggl(\frac
{v_0+t}{v_0} \biggr)^{-\beta}
\\
&&{}+ \int_{s=v_0}^{v_0+t} \biggl(\frac{v}{v_0}
\biggr)^{-\beta} \delta _{x_0+v_0+t-v}(dx) \frac{(v_0+t-s-v)^{\beta-1}}{\Gamma(\beta)}
\\
&&\hspace*{39pt}{}\times \1\{0< v<v_0+ t-s\} \frac{\beta s^{-\beta-1}}
{\Gamma(1-\beta)} \,ds \,dv. 
\end{eqnarray*}
%
Note that the above formulae extend Example~5.5 in \cite{coupleCTRW},
which calculates the law of $X_{t-}$.
The joint distribution of $\{(X_{t_i-},V_{t_i-})\dvtx 0\leq i\leq n\}$ can
now be computed by a simple conditioning argument.
Similarly, the semigroup for $(Y_t,R_t)$ reads
\begin{eqnarray*}
&&Q_t(y_0,r_0;dy,dr)
\\
&&\qquad=\delta_{y_0}(dy)\delta_{r_0-t}(dr)\1\{r_0\geq t
\} +Q_{t-r_0}(dy-y_0,dr)\1\{0<r_0< t\}
\\
&&\qquad=\delta_{y_0}(dy)\delta_{r_0-t}(dr)\1\{r_0\geq t
\}
\\
&&\qquad\quad{}+\1\{0<r_0< t\}\delta_{r+t-r_0+y_0}(dy)\\
&&\qquad\quad{}\times\int_{w=0}^{t-r_0}
\frac
{w^{\beta
-1}}{\Gamma(\beta)} \frac{\beta(t-r_0+r-w)^{-\beta-1}}{\Gamma(1-\beta)}\,dw \,dr. 
\end{eqnarray*}
The joint distributions of $X_{t-}$, $Y_{t}$ lead directly to the joint
distribution of CTRWL, OCTRWL, respectively, for a wide variety of
coupled models; see~\cite{OCTRW2}.
\end{example}

%


\begin{appendix}\label{app}
\section*{Appendix: Proofs}

\begin{lemmas} \label{lemmaOptional}
Let $T$ be $\mathcal H$-optional.
Then $E_T\dvtx \omega\mapsto E_{T(\omega)}(\omega)$
is $\mathcal F$-optional.
\end{lemmas}
\begin{pf}
We first assume that $T$ is single valued. That is, fix $t>0$ and $U
\in\mathcal H_t$, and let
$T(\omega) = t \cdot\1 \{ \omega\in U \} + \infty\cdot\1 \{\omega
\notin U \}$.
It is easy to check that $T$ is indeed $\mathcal H$-optional.
Now $\{E_T \le u\} = \{E_t \le u\} \cap U$, and the right-hand side
lies in $\mathcal F_u$, which follows from $U \in\mathcal H_t =
\mathcal F_{E(t)}$ and the definition of the stopped $\sigma$-algebra
$\mathcal F_{E(t)}$. Now consider an $\mathcal H$-optional time $T$
with countably many values $t_n$, so that $\Omega= \bigcup_{n \in
\mathbb N} \{\omega\dvtx  T(\omega) = t_n\}$.
Then due to the a.s. nondecreasing sample paths of $E$, we have
$E(\inf T_n) = \inf E(T_n)$, and an application of \cite{Kallenberg}, Lemma~6.3/4, together with the right-continuity of the
filtrations $\mathcal F$ and $\mathcal H$ shows that $E_T$ is $\mathcal
H$-optional.
\end{pf}

Stopping times allows for a decomposition into a predictable and
totally inaccessible part \cite{Kallenberg}.
The following lemma gives an interpretation for stopping times of the
form $E_T$.

\begin{lemmas} \label{decomposition}
Let $T > 0$ be an $\mathcal H$-predictable stopping time. Then the
$\mathcal F$-stopping time $E_T$ is predictable on the set
$\{\omega\dvtx  E_{T-\eps}(\omega) < E_T(\omega)\ \forall\eps> 0\} = \{
V_{T-} = 0\}$ and totally inaccessible on the complement
$\{\omega\dvtx  \exists\eps> 0, E_{T-\eps}(\omega) = E_T(\omega)\} = \{
V_{T-} > 0\}$.
Moreover, $\Delta(A,D)_{E_T}=(0,0)$ on $\{V_{T-}=0\}$ and\break  $\Delta
D_{E_T} > 0$ on $\{V_{T-}>0\}$, $\pr^{\chi,\tau}$-a.s.
\end{lemmas}
\begin{pf}
Let $T_n$ be an announcing sequence (\cite{Kallenberg}, page 410), for $T$,
that is $T_n$ are $\mathcal H$-stopping times, $T_n < T$, $T_n \uparrow
T$ a.s.
Then due to the a.s. continuity of sample paths of $E$, the sequence
$E_{T_n}$ announces $E_T$ on the set $\{V_{T-} = 0\}$, that is $E_T$ is
predictable on this set.
As a canonical Feller process, $(A,D)$ is quasi-left-continuous, and
all its jump times are totally inaccessible (\cite{Kallenberg},
Proposition~22.20), hence $\Delta(A,D)_{E_T}=(0,0)$, $\pr^{\chi
,\tau}$-a.s. on $\{V_{T-}=0\}$.
On the complementary set $\{V_{T-}> 0\}$, we have
$0 < H_T - G_{T-} = \Delta D_{E_T}$, and hence the process $D$ jumps at~$E_T$.
\end{pf}

\begin{pf*}{Proof of Theorem~\ref{th:RYembed}}
We first prove (ii).
Consider the set of $\omega$ such that $H_T(\omega) > t$.
In this case, $\mathbf M_\omega\cap(T,{t}) = \varnothing$, and hence
$E_T = E_{t}$, so $(Y_{t},H_{t}) = (Y_T,H_T)$, which implies that
%
\setcounter{equation}{0}
\begin{eqnarray}
\label{eq:Ds>t} 
\E^{\chi,\tau}\bigl[f(Y_{t},R_{t})
\1_{\{H_T>t\}}|\mathcal H_T\bigr] &=& f(Y_T,H_T-t)
\1_{\{H_T>t\}}
\nonumber
\\[-8pt]
\\[-8pt]
\nonumber
&=& f\bigl(Y_T,R_T -(t-T)\bigr) \1_{\{H_T>t\}}.
\end{eqnarray}
This corresponds to the first case in \eqref{evolutionQ}.
Turning to the second case, \mbox{$H_T(\omega) \le{t}$}, consider the shift
operators $\theta_t$ acting on $\Omega$, which are defined as usually
by $(\theta_t\omega)(u) = \omega(t+u)$, or equivalently
%
\begin{equation}
\label{eq:vartheta} (A,D)_u(\theta_t\omega) =
(A,D)_{t+u}(\omega),
\end{equation}
since $(A,D)$ is canonical for $\Omega$.
Then from the definition of the inverse process $E$, we find
%
\begin{eqnarray}
\label{eq:Eshift} 
E_t(\theta_{E_T}\omega) &=&
\inf\bigl\{u\ge0\dvtx D_u(\theta_{E_T}\omega) > t\bigr\} =
\inf\bigl\{u\ge0\dvtx D_{u + E_T}(\omega) > t \bigr\}\nonumber
\\
&=& \inf\bigl\{u\dvtx u-E_T(\omega) \ge0, D_u(\omega) >
t \bigr\} - E_T(\omega)
\\
&=&E_t(\omega) - E_T(\omega),\nonumber 
\end{eqnarray}
where $\theta_{E_T}\omega= \theta_{u}\omega$ if $E_T(\omega) = u$.
Now observe that $(A,D)_{E_t}$ is the point in $\mathbb{R}^{d+1}$
where the
process $(A,D)$ enters the set
$\mathbb{R}^d\times({t},\infty)$.
This point will be the same for the space--time path started at the
earlier time $E_T$, that is,
%
\begin{equation}
\label{eq:vartheta-inv} (A,D)_{E_t} \circ\theta_{E_T} =
(A,D)_{E_t}.
\end{equation}
In fact, using \eqref{eq:vartheta} and \eqref{eq:Eshift} we find
\begin{eqnarray*}
(A,D)_{E_t}(\theta_{E_T}\omega) &=& (A,D)_{E_t(\theta_{E_T}\omega)}(
\theta_{E_T}\omega) = (A,D)_{E_T(\omega) + E_t(\theta_{E_T}\omega)}(\omega)
\\
&=& (A,D)_{E_t(\omega)}(\omega) = (A,D)_{E_t}(\omega)
\end{eqnarray*}
for all $\omega\in\Omega$. Hence, we have shown that
\[
H_{t}(\theta_{E_T}\omega) = H_{t}(\omega),
\qquad
Y_{t}(\theta_{E_T}\omega) = Y_{t}(\omega)
\]
on the set $\{H_T\le{t}\}$. This yields
%
\begin{eqnarray}
\label{eq:Ds<t} 
\E^{\chi,\tau}\bigl[f(Y_{t},R_{t})
\1_{\{H_T \le t\}}|\mathcal H_T\bigr] &=&\E^{\chi,\tau}
\bigl[f(Y_{t},R_{t})\circ\theta_{E_T}
\1_{\{H_T \le t\}
}|\mathcal H_T\bigr]
\nonumber\\
&=&\E^{\chi,\tau}\bigl[f(Y_{t},R_{t}) \circ
\theta_{E_T}|\mathcal F_{E_T}\bigr] \1 _{\{H_T \le t\}}
\nonumber
\\[-8pt]
\\[-8pt]
\nonumber
&=&\E^{(A,D)_{E_T}}\bigl[f(Y_{t},R_{t})\bigr]
\1_{\{H_T \le t\}}
\\
&=&\E^{Y_T,H_T}\bigl[f(Y_{t},R_{t})\bigr]
\1_{\{H_T \le t\}}\nonumber 
\end{eqnarray}
$\pr^{\chi,\tau}$-almost surely, using the strong Markov property of
$(A,D)$ at the stopping time $E_T$.
Then (ii) follows by adding equations \eqref{eq:Ds>t} and \eqref{eq:Ds<t}.

As for (i), let $(y_0,r_0) \in\mathbb{R}^d\times[0,\infty)$.
Then $\pr^{y_0,q+r_0}(Y_r=y_0,R_q = r_0)=1$, and hence by nested
conditional expectations and the above calculations we have
\begin{eqnarray*}
Q_{q,t}f(y_0,r_0) &=& \E^{y_0,q+r_0}
\bigl[Q_{q,t}f(Y_q,R_q)\bigr]
\\
&=& \E^{y_0,q+r_0}\bigl[\E^{y_0,q+r_0}\bigl[f(Y_t,R_t)|
\mathcal H_q\bigr]\bigr]
\\
&=& \E^{y_0,q+r_0}\bigl[\E^{y_0,q+r_0}\bigl[\E^{y_0,q+r_0}
\bigl[f(Y_t,R_t)|\mathcal H_s\bigr]|
\mathcal H_q\bigr]\bigr]
\\
&=&\E^{y_0,q+r_0}\bigl[\E^{y_0,q+r_0}\bigl[Q_{s,t}f(Y_s,R_s)|
\mathcal H_q\bigr]\bigr]
\\
&=& \E^{y_0,q+r_0}\bigl[Q_{q,s}Q_{s,t}f(Y_q,R_q)
\bigr] = Q_{q,s}Q_{s,t}f(y_0,r_0).
\end{eqnarray*}
We turn to the remaining case (iii).
By definition of $R_t$, it suffices to show that if $T$ is a $\mathcal
H$-predictable time, then $(Y,H)_{T-} = (Y,H)_T$, $\pr^{\chi,\tau
}$-a.s. for every $(\chi,\tau) \in\mathbb{R}^{d+1}$.
Hence let $T_n < T$, $T_n \uparrow T$ be a sequence of $\mathcal
H$-optional times announcing~$T$.
As in Lemma~\ref{decomposition}, we check the two cases in which the
$\mathcal F$-stopping time $E_T$ is predictable or totally inaccessible.

On the set
$\{\omega\dvtx  E_{T-\eps}(\omega) < E_T(\omega), \forall\eps> 0\}$, the
process $E$ is left-increasing at $T$, continuous, and $E_{T_n}
\uparrow E_T$, $E_{T_n} < E_T$ if
$T_n \uparrow T$, $T_n < T$. Moreover, $\Delta(A,D)_{E_T} = (0,0)$
a.s. (Lemma~\ref{decomposition}). Hence,
\[
(H,Y)_{T-} = (A,D)_{E_{T-}} = (A,D)_{E_T-} =
(A,D)_{E_T} = (H,Y)_{T}.
\]

On the set $\{\omega\dvtx  \exists\eps> 0 \dvtx E_{T-\eps}(\omega) =
E_T(\omega
)\}$, $E$ is left-constant at $T$. Hence, $E_{T_n} = E_T$ for large
$n$, and
\[
(H,Y)_{T-} = \lim(H,Y)_{T_n} = \lim(A,D)_{E_{T_n}} =
(A,D)_{E_T} = (H,Y)_{T}.
\]
The two cases together imply that $(H,Y)_{T-} = (H,Y)_{T}$ a.s.
\end{pf*}

For the proof of Theorem~\ref{th:AXembed}, we will need the following lemma.

\begin{lemmas} \label{lem:jumps-ex}
Let $(\chi,\tau)\in\mathbb{R}^{d+1}$,
and let $t \ge\tau$.
Then for every bounded measurable $f$ defined on $\mathbb
{R}^{d+1}\times
\mathbb{R}^{d+1}
$, we have $\pr^{\chi,\tau}$-a.s.:
\begin{eqnarray*}
&&\E^{\chi,\tau}\bigl[f\bigl(X_{t-},G_{t-};
\Delta(A,D)_{E(t)}\bigr)|\mathcal G_t\bigr]
\\
&&\qquad= \int_{\mathbb{R}^{d+1}} K_{V_{t-}}(X_{t-},G_{t-};dx
\times dz) f(X_{t-},G_{t-};x,z).
\end{eqnarray*}
\end{lemmas}

\begin{pf}
Since $(X_{t-},G_{t-})$ are $\mathcal G_t$-measurable, by a monotone
class argument and dominated convergence, it suffices to prove the formula
%
\begin{equation}
\label{lemmalemma} \E^{\chi,\tau}\bigl[f\bigl(\Delta(A,D)_{E(t)}\bigr)|
\mathcal G_t\bigr] = \int_{\mathbb{R}^{d+1}}
K_{V_{t-}}(X_{t-},G_{t-};f)
\end{equation}
for all bounded measurable $f$ defined on $\mathbb{R}^{d+1}$.
As in Lemma~\ref{decomposition}, we consider the two cases
$\{V_{t-}=0\}$ and $\{V_{t-}>0\}$.
On $\{V_{t-}=0\}$, we have $\Delta(A,D)_{E_t}=(0,0)$, $\pr^{\chi
,\tau
}$-a.s., and hence
%
\begin{eqnarray}
\label{eq:A-=0} 
&&\E^{\chi,\tau}\bigl[f(\Delta(A,D)_{E(t)}
\1\{V_{t-}=0\} |\mathcal G_t\bigr]
\nonumber
\\[-8pt]
\\[-8pt]
\nonumber
&&\qquad=f(0,0)
= \delta_{(0,0)}(f) = K_{V_{t-}}(X_{t-},G_{t-};f)
\1\{V_{t-}=0\}. 
\end{eqnarray}

On $\{V_{t-} > 0\}$, the process $D$ jumps at $E_t$ (Lemma~\ref
{decomposition}), and since $D$ has increasing sample paths this is
equivalent to
%
\begin{equation}
\label{crossover} \mbox{``there exists a unique number } s >0 \mbox{ such that }
D_{s-} < t \le D_s\mbox{.''}
\end{equation}
We rewrite the restriction of \eqref{lemmalemma} to $\{V_{t-} > 0\}$ in
integral form:
\begin{eqnarray*}
&&\E^{\chi,\tau}\bigl[f\bigl(\Delta(A,D)_{E(t)}\bigr)
\1_C \1\{V_{t-} > 0\}\bigr]
\\
&&\qquad= \E^{\chi,\tau}\bigl[ K_{V_{t-}}(X_{t-},G_{t-};f)
\1_C\1\{V_{t-} > 0\}\bigr], \qquad C \in\mathcal
G_t,
\end{eqnarray*}
where $\1_C(\omega) = 1$ iff $\omega\in C$.
Now we invoke
\cite{DM79}, Theorem IV.67(b), which says that there exists an
$\mathcal
F$-adapted \emph{predictable} process $Z$ such that $\1_C = Z_{E(t)}$.
Then it suffices to show that for every $\mathcal F$-adapted
predictable process $Z$, the following two random variables have the
same expectation with respect to $\pr^{\chi,\tau}$:
%
\begin{eqnarray}
\label{LHS}&& f\bigl(\Delta(A,D)_{E(t)}\bigr) Z_{E(t)} \1
\{V_{t-} > 0\},
\nonumber
\\[-8pt]
\\[-8pt]
\nonumber
&& K_{V_{t-}}f(X_{t-},G_{t-}) Z_{E(t)} \1
\{V_{t-} > 0\}.
\end{eqnarray}
We begin on the right-hand side and find, using \eqref{crossover} and
$X_{t-} = A_{E_t-}$, $G_{t-} = D_{E_t-}$
\begin{eqnarray*}
& &\E^{\chi,\tau} \bigl[ K_{V_{t-}}(X_{t-},G_{t-};f)Z_{E_t}
\1\{V_{t-}>0\} \bigr]
\\
&&\qquad= \E^{\chi,\tau} \bigl[ K_{t-D_{E_t-}}(A_{E_t-},D_{E_t-};f)Z_{E_t}
\1\{D_{E_t-} < t\} \bigr]
\\
&&\qquad=\E^{\chi,\tau} \biggl[ \sum_{s>0}
K_{t-D_{s-}}(A_{s-},D_{s-};f)Z_s \1
\{D_{s-}<t \le D_s\} \biggr]
\\
&&\qquad=\E^{\chi,\tau} \biggl[\sum_{s>0}
K_{t-D_{s-}}(A_{s-},D_{s-};f)Z_s\1
\{D_{s-}<t\} \1\{\Delta D_s \ge t-D_{s-}\}
\biggr]
\\
&&\qquad= \E^{\chi,\tau} \bigl[W(\cdot,s;y,w) \mu(\cdot,ds;dy,dw) \bigr] = \cdots,
\end{eqnarray*}
where the optional random measure $\mu$ is as in \eqref
{randomMeasure} and
\begin{eqnarray*}
W(\omega,s;y,w)
&=& K_{t-D_{s-}(\omega)} \bigl(A_{s-}(\omega), D_{s-}(
\omega);f\bigr)Z(s,\omega)
\nonumber
\\[-8pt]
\\[-8pt]
\nonumber
&&{}\times\1\bigl\{D_{s-}(\omega) < t\bigr\} \1\bigl
\{w \ge t-D_{s-}(\omega)\bigr\}
\end{eqnarray*}
is a predictable integrand.
The compensation formula \cite{JacodShiryaev}, II.1.8, and \eqref
{compensator} then yield
\begin{eqnarray*}
\cdots&=& \E^{\chi,\tau}\bigl[W(\cdot,s;y,w) \mu^p(\cdot,ds;dy,dw)
\bigr]
\\
&= &\E^{\chi,\tau} \biggl[\int_0^\infty
K_{t-D_{s-}} (A_{s-},D_{s-};f)Z_s\1
\{D_{s-} < t\}\\
&&\hspace*{45pt}{} \times K \bigl(A_{s-},D_{s-};\mathbb{R}^d
\times(t-D_{s-},\infty) \bigr)\,ds \biggr].
\end{eqnarray*}
Using the definition of $K_v$ \eqref{eq:K_a}, this equals
%
\begin{eqnarray}\label{RHSend}
&=& \E^{\chi,\tau} \biggl[\int_0^\infty
\int_{\mathbb{R}^d\times
[t-D_{s-},\infty)} K(A_{s-},D_{s-};dy,dw) f(y,w)
Z_s\nonumber\\
&&\hspace*{183pt}{}\times \1\{D_{s-} <t\} \,ds \biggr]
\nonumber
\\[-8pt]
\\[-8pt]
\nonumber
 &=& \E^{\chi,\tau} \biggl[\int_0^\infty
\int_{\mathbb{R}^{d+1}} K(A_{s-},D_{s-};dy,dw)f(y,w)
Z_s \\
&&\hspace*{90pt}{}\times \1\{D_{s-}<t\le D_{s-}+w\}\,ds \biggr].\nonumber
\end{eqnarray}
Proceeding similarly with the left-hand side of \eqref{LHS},
we find
\begin{eqnarray}\label{LHSend}
\nonumber
&& \E^{\chi,\tau} \bigl[f\bigl(\Delta(A,D)_{E(t)}\bigr)
Z_{E(t)}\1\{ V_{t-}>0\} \bigr],
\\
&&\qquad= \E^{\chi,\tau} \biggl[\sum_{s >0}f
\bigl(\Delta(A,D)_s\bigr)Z_s \1\{D_{s-}<t\le
D_{s-}+\Delta D_s\} \biggr]
\nonumber
\\[-8pt]
\\[-8pt]
\nonumber
\nonumber
&&\qquad= \E^{\chi,\tau}\bigl[\tilde W(\cdot,s;y,w) \mu(\cdot,ds;dy,dw)
\bigr]
\\
 &&\qquad= \E^{\chi,\tau}\bigl[\tilde W(\cdot,s;y,w)
\mu^p(\cdot,ds;dy,dw)\bigr],\nonumber
\end{eqnarray}
where $\tilde W(\omega,s;y,w) = f(y,w) Z_s(\omega) \1\{D_{s-}(\omega
) <
t \le D_{s-}(\omega) + w\}$.
We check that \eqref{LHSend} and \eqref{RHSend} are equal.
Hence, we have shown
%
\begin{eqnarray}
\label{eq:A->0} &&\E^{\chi,\tau} \bigl[f\bigl(\Delta(A,D)_{E(t)}\bigr) \1
\{V_{t-} > 0\}|\mathcal G_t \bigr]
\nonumber
\\[-8pt]
\\[-8pt]
\nonumber
&&\qquad= K_{V_{t-}}f(X_{t-},G_{t-})
\1\{V_{t-} > 0\},
\end{eqnarray}
and adding equations \eqref{eq:A-=0} and \eqref{eq:A->0} yields
\eqref
{lemmalemma}.
\end{pf}

For later use, we note the formula
%
\begin{eqnarray}
\label{eq:kdarts} 
&&K_{v+t}(x,z;C) K_v\bigl(x,z;
\mathbb{R}^d\times[v+t,\infty)\bigr)
\nonumber
\\[-8pt]
\\[-8pt]
\nonumber
&&\qquad= K_v(x,z;C)\qquad (x,z)\in\mathbb{R}^{d+1}, v,t \ge0,
\end{eqnarray}
valid for all Borel-sets $C \subset\mathbb{R}^d\times[v+t,\infty)$.

\begin{pf*}{Proof of Theorem~\ref{th:AXembed}}
We begin with statement (ii).
We consider the two cases $H_s \ge t$ and $H_s < t$.
On the set $\{H_s \ge t\}$, $E$ is constant on the interval $[s,t]$,
and hence we have $(G,X)_{t-} = (G,X)_{s-}$.
Using $G_{s-} + \Delta D_{E_s} = H_s$ and Lemma~\ref{lem:jumps-ex}, we
calculate
%
\begin{eqnarray}
\label{eq:Ds>t2} 
&&\E^{\chi,\tau}\bigl[f(X_{t-},V_{t-})
\1\{H_s\ge t\}|\mathcal G_s\bigr]\nonumber
\\
&&\qquad= f(X_{s-},t-G_{s-}) \pr^{\chi,\tau}(H_s
\ge t|\mathcal G_s)\nonumber
\\
&&\qquad= f(X_{s-},t-s+V_{s-}) \pr^{\chi,\tau}(\Delta
D_{E_s}\ge t - G_{s-} |\mathcal G_s)
\\
&&\qquad=f(X_{s-},t-s+V_{s-}) K_{V_{s-}}
\bigl(X_{s-},G_{s-};[t-G_{s-},\infty )\times
\mathbb{R}^d\bigr)\nonumber
\\
&&\qquad=f(X_{s-},V_{s-}+t-s) K_{V_{s-}}
\bigl(X_{s-},s-V_{s-};[t-s+V_{s-},\infty )
\times\mathbb{R}^d\bigr),\nonumber 
\end{eqnarray}
which corresponds to the first summand in \eqref{evolutionP}.

We now turn to the case $H_s < t$, and recall the shift operators
$\theta_t$ from \eqref{eq:vartheta}.
For the left-continuous version of $(A,D)$, we can write
\[
(A,D)_{t-}(\theta_s\omega) = (A,D)_{s+t-}(
\omega),\qquad s \ge0, t >0.
\]
Note that we had to assume $t >0$ above, for the left-hand limit to be defined.
We find now, similarly to \eqref{eq:vartheta-inv},
\[
(A,D)_{E_t-} \circ\theta_{E_s} = (A,D)_{E_t-}
\]
on $\{H_s < t\}$.
Indeed, by \eqref{eq:Eshift}, $E_t(\omega) = E_s(\omega) +
E_t(\theta
_{E_s}\omega)$, and so
\begin{eqnarray*}
(A,D)_{E_t-}(\theta_{E_s}\omega) &=& (A,D)_{E_t(\theta_{E_s}\omega)-}(
\theta_{E_s}\omega) = (A,D)_{E_s(\omega)+E_t(\theta_{E_s}\omega)-}(\omega)
\\
&= &(A,D)_{E_t(\omega)-}(\omega) = (A,D)_{E_t-}(\omega).
\end{eqnarray*}
If $t > 0 $ and $H_s(\omega)< t$, then by \eqref{eq:Eshift}
$E_t(\theta_{E_s}\omega) = E_t(\omega) - E_s(\omega) > 0$, and the
left-hand limit is well defined.
Thus, we have shown that on the set $\{H_s<t\}$ we have
$(X_{t-},V_{t-}) = (X_{t-},V_{t-}) \circ\theta_{E_s}$.
We will use the strong Markov property of $(A,D)$ in the following form:
\[
\E^{\chi,\tau}[F \circ\theta_T | \mathcal F_T] =
\E^{A_T,D_T}[F], \qquad \pr^{\chi,\tau}\mbox{-a.s.}, 
\]
valid for all $\mathcal F$-stopping times $T$ and random variables
$F$ on $(\Omega, \mathcal F_\infty, \pr^{\chi,\tau})$. Using Lemma~\ref
{lem:jumps-ex} and the strong Markov property at $E_s$, we then calculate
%
\begin{eqnarray}
\label{eq:D<t2}&& \E^{\chi,\tau} \bigl[f(X_{t-},V_{t-})\1
\{H_s<t\}|\mathcal G_s \bigr]\nonumber
\\
&&\qquad= \E^{\chi,\tau} \bigl[\E^{\chi,\tau}\bigl[f(X_{t-},V_{t-})
\circ \theta _{E_s}|\mathcal H_s\bigr]\1
\{H_s<t\}|\mathcal G_s \bigr]\nonumber
\\
&&\qquad= \E^{\chi,\tau} \bigl[\E^{Y_s,H_s}\bigl[f(X_{t-},V_{t-})
\bigr]\1\{H_s<t\} |\mathcal G_s \bigr]\nonumber
\\
&&\qquad= \E^{\chi,\tau} \bigl[\E^{(X,G)_{s-}+\Delta(A,D)_{E_s}} \bigl[f(X_{t-},V_{t-})
\bigr]\1\{G_{s-} + \Delta D_{E_s}<t\}|\mathcal
G_s \bigr]
\nonumber
\\
&&\qquad= \int_{\mathbb{R}^{d+1}} K_{V_{s-}}(X_{s-},G_{s-};dy
\times dw)
\\
&&\qquad\quad{}\times \E^{(X,G)_{s-}+(y,w)} \bigl[f(X_{t-},V_{t-})\bigr]\1
\{G_{s-}+w<t\}\nonumber
\\
&&\qquad= \int_{ \mathbb{R}^d\times[V_{s-}, V_{s-}+t-s) } K_{V_{s-}}(X_{s-},
s-V_{s-};dy\times dw)\nonumber
\\
&&\qquad\quad{}\times \E^{X_{s-}+y,s-V_{s-}+w} \bigl[f(X_{t-},V_{t-})\bigr],\nonumber
\end{eqnarray}
which corresponds to the second summand in \eqref{evolutionP}.
Adding equations \eqref{eq:Ds>t2} and~\eqref{eq:D<t2} yields
statement (ii).
For statement (i), we calculate
\begin{eqnarray*}
&&P_{r,s}P_{s,t}f(x,v)\\
&&\qquad = P_{s,t}f(x,v+s-r)
K_v\bigl(x,r-v;\mathbb{R}^d\times[{v+s-r},\infty)\bigr)
\\
&&\qquad\quad+ \int_{\mathbb{R}^d\times[v,v+s-r)} K_v(x,r-v;dy\times dw)
\E^{x+y,r-v+w}\bigl[P_{s,t}f(X_{s-},V_{s-})
\bigr]
\\
&&\qquad= K_v \bigl(x,r-v;\mathbb{R}^d\times[v+s-r,\infty
) \bigr)
\\
&&\qquad\quad{}\times \biggl\{f(x,v+t-r) K_{v+s-r}\bigl(x,r-s;\mathbb{R}^d\times
[{v+t-r},\infty)\bigr)
\\
&&\hspace*{18pt}\qquad\quad{} + \int_{\mathbb{R}^d\times[v+s-r,v+t-r)} K_{v+s-r}(x,r-v;dy\times dw)
\\
&&\hspace*{126pt}\qquad\quad{}\times\E^{x+y,r-v+w}\bigl[f(X_{t-},V_{t-})\bigr] \biggr\}
\\
&&\qquad\quad{}+\int_{\mathbb{R}^d\times[v,v+s-r)} K_v(x,r-v;dy\times dw)
\E^{x+y,r-v+w} \bigl[P_{s,t}f(X_{s-},V_{s-})
\bigr] \\
&&\qquad= \cdots.
\end{eqnarray*}
Using \eqref{eq:kdarts} and applying the statement (ii) with
$(\chi,\tau) = (x+y,r-v+w)$ yields
\begin{eqnarray*}
\cdots&=& f(x,v+t-r) K_v\bigl(x,r-v;\mathbb{R}^d
\times[v+t-r,\infty)\bigr)
\\
&&{}+ \int_{\mathbb{R}^d\times[v+s-r,v+t-r)} K_v(x,r-v;dy\times dw)
\E^{x+y,r-v+w}\bigl[f(X_{t-},V_{t-})\bigr]
\\
&&{}+ \int_{\mathbb{R}^d\times[v, v+s-r)} K_v(x,r-v;dy\times dw)
\\
&&\hspace*{40pt}\qquad\quad{}\times\E^{x+y,r-v+w} \bigl[\E^{x+y,r-v+w}\bigl[f(X_{t-},V_{t-})|
\mathcal G_s\bigr] \bigr]
\\
&=& P_{r,t}f(x,v),
\end{eqnarray*}
which is statement (i).
\end{pf*}
\end{appendix}

%

%


\printaddresses

\end{document}